\documentclass[10pt,a4paper]{article}

\usepackage{url}
\usepackage{verbatim}
\usepackage{latexsym}
\usepackage{amssymb,amstext,amsmath,amsthm}
\usepackage{epsf}
\usepackage{epsfig}

\usepackage{tikz-cd}
\usepackage{tikz}
\usetikzlibrary{shapes.geometric}
\usepackage{a4wide}
\usepackage{verbatim}
\usepackage{proof}
\usepackage{latexsym}
\newtheorem{theorem}{Theorem}[section]

\newtheorem{lemma}{Lemma}[section]
\newtheorem{proposition}{Proposition}[section]
\newtheorem{definition}[theorem]{Definition}
\newtheorem{remark}[theorem]{Remark}

\usepackage{hyperref}
\usepackage{cleveref}

\newcommand{\Gr}{\mathrm{Gr}}

\DeclareMathOperator{\PGL}{PGL}
\DeclareMathOperator{\GL}{GL}

\DeclareMathOperator{\Spec}{Spec}

\newcommand{\SB}{\mathrm{SB}}
\newcommand{\RI}{\mathrm{RI}}
\newcommand{\AZ}{\mathrm{AZ}}
\newcommand{\propTrunc}[1]{\lVert #1 \rVert}
\newcommand{\bP}{\mathbb{P}}
\newcommand{\N}{\mathbb{N}}
\newcommand{\fib}{\mathrm{fib}}
\newcommand{\Aut}{\mathrm{Aut}}

\usepackage{authblk}
\usepackage{doi}

\begin{document}

\title{Ch\^atelet's Theorem in Synthetic Algebraic Geometry}

\author{Thierry Coquand}
\affil{Computer Science and Engineering Department,
University of Gothenburg, Sweden,
\url{coquand@chalmers.se}
}
\author{Hugo Moeneclaey}
\affil{Computer Science and Engineering Department,
University of Gothenburg, Sweden,
\url{hugomo@chalmers.se}
}
\date{\today}
\maketitle



\section*{Introduction}

{F}ran\c cois {C}h\^atelet introduced the notion of Severi-Brauer variety in his 1944 PhD thesis
\cite{chatelet44}. One motivation is to provide
a generalisation of the well-known result that a conic which has a rational point is isomorphic to $\bP^1(k)$.
He defines a Severi-Brauer variety to be a variety which becomes isomorphic to some $\bP^n(k)$ after
a separable extension. After recalling the characterisation of a central simple algebra over a field $k$, as
an algebra which becomes isomorphic to a matrix algebra $M_n(k)$ after a separable extension, he notices the fundamental
fact that, $\bP^{n}(k)$ and $M_{n+1}(k)$ have the same automorphism group $PGL_{n+1}(k)$. He uses then this
to describe a correspondence between Severi-Brauer varieties and central simple algebras, and as a corollary
obtains the following generalisation of Poincar\'e's result: a Severi-Brauer variety which has a rational point
is isomorphic to some $\bP^{n}(k)$. This result and its proof are described in Serre's book on local fields \cite{serre62}.
(The paper \cite{colliot88} and the book \cite{gille2017central} also contain a description of this result.)

The notion of central simple algebra over a field
has been generalised to the notion of Azumaya algebra  \cite{azumaya51}, and
Grothendieck \cite{grothendieck68} generalized the notion of Severi-Brauer over over an arbitrary comuutative ring.
The goal of this note is to present a formulation and proof of Ch\^atelet's Theorem over an arbitrary commutative ring
in the setting of synthetic algebraic geometry \cite{draft}, using the results already proved about projective
space \cite{sag-projective} in this context, in particular the fact that any automorphism of the projective space is given
by a homography. We also rely essentially on basic results about dependent type theory with univalence \cite{hott}
and modalities \cite{modalities}, in particular the fact that, in this context, \'etale sheafification can be described
as modalities. The formulation of Ch\^atelet's Theorem becomes that for $X$ a scheme, we have that:
\[\propTrunc{X=\bP^n}_{T} \to \propTrunc{X} \to \propTrunc{X=\bP^n}\]
where:
\[\propTrunc{X=\bP^n}_{T}\]
is the localisation of $\propTrunc{X=\bP^n}$ for a modality $T$ satisfying some basic properties (valid for the modality
corresponding to \'etale sheafification).


\section{\'Etale sheaves}
\label{etale-sheaves}

\subsection{Affine schemes are \'etale sheaves}



Monic unramifiable polynomials are defined in \cite{wraith79} and analysed in \cite{coqazumaya}.
If $P$ is a proposition, we say that $A$ is $P$-local if, and only if, the canonical map $A\rightarrow A^P$ is an equivalence.
Given a family of propositions $P_i$, the types that are $P_i$-local for all $i$ form a model of type theory with univalence,
and we have an associated lex modality, the nullification modality \cite{modalities,Quirin16}.
Following \cite{wraith79}, we can consider the \'etale modality, which corresponds to the family of propositions
 $\propTrunc{\Spec(R[X]/g)}$, for $g$ monic unramifiable.

\begin{definition}
A type $X$ is called an \'etale sheaf if for all $g:R[X]$ monic unramifiable, we have that $X$ is $\propTrunc{\Spec(R[X]/g)}$-local.
\end{definition}


\begin{remark}
  By \cite{wraith79} this should agree with the usual \'etale topology.
  It should also be noted that we will never use the unramifiability assumption, so we could just use non-constant monic polynomials instead.
\end{remark}

\begin{lemma}\label{etale-subcanonical}
The type $R$ is an \'etale sheaf.
\end{lemma}

\begin{proof}
Let $g:R[X]$ be monic and write $S=\Spec(R[X]/g)$. We have a coequaliser in sets:
\[S\times S\rightrightarrows S \to \propTrunc{S}\]
So since $R$ is a set we have an equaliser diagram:
\[R^{\propTrunc{S}} \to R^S\rightrightarrows R^{S\times S}\]
so that it is enough to prove that $R$ is the equaliser of:
\[R[X]/g \rightrightarrows R[X]/g \otimes R[X]/g\]
to conclude. But since $g$ is monic we merely have:
\[R[X]/g \simeq R^n\]
and it is clear that $R$ is the equaliser of:
\[R^n \rightrightarrows R^n\otimes R^n\]
\end{proof}

\begin{remark}\label{R-modal-subcanonical}
If $R$ is modal, then so is $\mathrm{Hom}(A,R)$ for any $R$-algebra $A$ by general reasoning on modalities, so that every affine scheme is modal. By duality this implies that every finitely presented algebra is modal.
\end{remark}

\subsection{Schemes are \'etale sheaves}

\begin{lemma}\label{scheme-are-sheaf-from-affine}
Assume given a proposition $P$ such that:
\begin{itemize}
\item The type $R$ is $P$-local.
\item Any open proposition is $P$-local.
\item The type of open propositions is $P$-local.
\end{itemize}
Then any scheme is $P$-local.
\end{lemma}

\begin{proof}
Since $R$ is $P$-local, all affine schemes are $P$-local as explained in \Cref{R-modal-subcanonical}.

We check that for all scheme $X$, any map:
\[f:P\to X\]
merely factors through $1$. Take $(U_i)_{i:I}$ a finite cover of $X$ by affine scheme. Then for any $i:I$ we have that $f^{-1}(U_i)$ is an in $P$, so since the type of open is $P$-local, we merely have an open proposition $V_i$ such that for all $x:P$, we have:
\[(x\in f^{-1}(U_i) )\leftrightarrow V_i\]
Since the $f^{-1}(U_i)$ cover $P$, we have that:
\[P\to \lor_{i:I} V_i\]
But open propositions are assumed to be $P$-local, so we have that:
\[ \lor_{i:I} V_i\]
Assume $k:I$ such that $V_k$ holds. Then $f^{-1}(U_k) = P$ and the map $f$ factors through the affine scheme $U_k$. Since affine schemes are $P$-local, we merely have a lift for $f$.

Now we conclude that any scheme is $P$-local by proving that its identity types are $P$-local. Indeed they are schemes, so the previous point implies they are $P$-local.
\end{proof}


We will use freely the terminology and results of \cite{draft}; in particular a proposition is {\em open}
if, and only if, it is equivalent to a proposition of the form $r_1\neq 0\vee\dots\vee r_m\neq 0$ for some
$r_1,\dots,r_m$ in $R$.

\begin{lemma}\label{roots-monic-proper}
  If $g$ is a monic polynomial, and $h_1,\dots,h_m$ in $R[X]$, then the proposition
  $\forall_xg(x)=0\rightarrow h_1(x)\neq 0\vee\dots\vee h_m(x)\neq 0$ is open.
  It follows that, for any monic $g:R[X]$, and for any open $U$ in $\Spec(R[X]/g)$ the proposition:
  \[\prod_{x:\Spec(R[X]/g)}U(x)\]
is open.
\end{lemma}

\begin{proof}
  This follows from  IV-10-2 in \cite{lombardi-quitte}.
\end{proof}

\begin{proposition}\label{scheme-is-etale-sheaf}
Any scheme is an \'etale sheaf.
\end{proposition}

\begin{proof}
Assume given $g:R[X]$ monic, we can apply \Cref{scheme-are-sheaf-from-affine} because:
\begin{itemize}
\item The type $R$ is an \'etale sheaf by \Cref{etale-subcanonical}.
\item Any open proposition $U$ is an \'etale sheaf because if:
\[\propTrunc{\Spec(R[X]/g)}\to U\]
then since $\neg\neg\Spec(R[X]/g)$ we have $\neg\neg U$, which implies $U$.
\item Since open propositions are \'etale sheaves, it is enough that any map:
\[\propTrunc{\Spec(R[X]/g)}\to \mathrm{Open}\]
merely factors through $1$. But given a constant open $U$ in $\Spec(R[X]/g)$, for any $x:\Spec(R[X]/g)$ we have that:
\[x\in U \leftrightarrow \prod_{y:\Spec(R[X]/g)} y\in U)\]
The right hand side is open by \Cref{roots-monic-proper}, giving the required lift. 
\end{itemize}
\end{proof}

\subsection{Descent for finite free modules}


\begin{lemma}\label{fp-equivalent-pointwise}
  If we have $M_x$ a finitely presented (resp. finite projective)
  $R$-module depending on $x:\Spec(A)$, then $\prod_{x:\Spec(A)}M_x$ is a finitely presented (resp. finite projective) $A$-module.
\end{lemma}

\begin{proof}
See Theorem 7.2.3 in \cite{draft}.
\end{proof}

\begin{lemma}\label{descent-sqc-etale}
Let $M$ be a module that is an \'etale sheaf such that we have the \'etale sheafification of "$M$ is f.p.", then for any monic $g$ we have that:
\[R[X]/g\otimes M \simeq M^{\Spec(R[X]/g)}\]
\end{lemma}

\begin{proof}
We have that $R[X]/g\otimes M$ is merely equal to $M^n$ where $n$ is the degree of $g$, therefore it is an \'etale sheaf. As $M^{\Spec(R[X]/g)}$ is an \'etale sheaf as well, so when proving that:
\[R[X]/g\otimes M \to M^{\Spec(R[X]/g)}\]
is an equivalence, we can assume that $M$ is finitely presented. In this case we conclude by Theorem 7.2.3 in \cite{draft}.
\end{proof}

\begin{lemma}\label{fp-stable-etale-tensor}
  Let $A$ be an fppf algebra and let $M$ be an $R$-module. Then if $A\otimes M$ is f.p. (resp. finite projective)
  as an $A$-module if and only if $M$ is f.p. (resp. finite projective) as an $R$-module.
\end{lemma}

\begin{proof}
See VIII.6.7 in \cite{lombardi-quitte}.
\end{proof}

\begin{proposition}\label{descent-finite-free}
For $M$ a module that is an \'etale sheaf, the proposition "$M$ is an finite free" is itself an \'etale sheaf.
\end{proposition}

\begin{proof}
  Follows from \Cref{descent-sqc-etale,fp-stable-etale-tensor}.
\end{proof}

\section{Azumaya algebras and their associated Severi-Brauer variety}
From now on we assume a lex modality $T$ such that:
\begin{itemize}
\item Schemes are modal.
\item If $M$ is a $T$-modal $R$-module, then the proposition of $M$ being finite free is modal.
\end{itemize}
We call $T$-modal types sheaves and we write $\propTrunc{X}_T$ the sheafification of the propositional truncation of $X$. Note that $T$-modal types form a model of homotopy type theory \cite{modalities,Quirin16}.

In \Cref{etale-sheaves} we constructed such a modality (\Cref{scheme-is-etale-sheaf} and \Cref{descent-finite-free}).

We fix a natural number $n$ throughout.

\subsection{The type $\AZ_n$ of Azumaya algebras}

\begin{definition}
An Azumaya algebra of rank $n$ is a (non-commutative, unital) $R$-algebra $A$ such that its underlying type is a sheaf and:
\[\propTrunc{A=M_{n+1}(R)}_T\]
\end{definition}

We write $\AZ_n$ for the type of Azumaya algebra of rank $n$.

\begin{remark}\label{azumaya-independent-modality}
In \cite{coqazumaya}, we give a constructive proof that a $R$-algebra $A$ is an Azumaya algebra of rank $n$
if, and only if, $A$ is free as a $R$-module of rank $(n+1)^2$ and the canonical map
$A\otimes A^{op}\rightarrow \mathrm{End}_R(A)$ is an isomorphism.
\end{remark}

\begin{lemma}\label{azumayas-are-finite-free}
For all $A:\AZ_n$ we have that $A$ is finite free as a module.
\end{lemma}

\begin{proof}
By hypothesis $A$ being finite free is modal so that $\propTrunc{A=M_{n+1}(R)}_T$ implies $A$ finite free.
\end{proof}

\begin{definition}
Let $V$ be a free $R$-module, we define $\Gr_k(V)$ the $k$-Grassmannian of $V$ as the type of $k$-dimensional subspaces of $V$.
\end{definition}

\begin{lemma}\label{grassmanians-are-schemes}
Let $V$ be a finite free module, then $\Gr_k(V)$ is a scheme.
\end{lemma}

\begin{proof}
We can assume $V=R^n$. The type of $k$-dimensional subspaces of $R^n$ is the type of $n\times k$ matrices of rank $k$ quotiented by the natural action of $\GL_k$. For all $k\times k$ minor, we consider the open proposition stating this minor is non-zero, which well defined as it is invariant under the $\GL_k$-action. This gives a finite open cover of $\Gr_k(R^n)$.

Let us show any piece is affine. For example consider the piece of matrices of the form:
\[\begin{pmatrix}
P & N
\end{pmatrix}\]
where $P$ is invertible of size $k\times k$. Any orbit in this piece has a unique element of the form:
\[\begin{pmatrix}
I_k & N'
\end{pmatrix}\]
where $I_k$ is the identity matrix, so this piece is equivalent to $R^{(n-k)k}$.
\end{proof}

\begin{lemma}\label{being-ideal-in-azumaya-closed}
For all $A:\AZ_n$ and $I:\Gr_{n+1}(A)$, we have that $I$ being a right ideal in $A$ is a closed proposition.
\end{lemma}

\begin{proof}
By \Cref{azumayas-are-finite-free} we have that $A$ is finite free as a module. Consider $a_0,\cdots,a_n$ a basis of $I$ and extend it to a basis of $A$ adding $b_1,\cdots,b_l$. We can proceed as if $R$ was a field because a non zero vector has a non-zero coefficient \cite{draft}.

For any $a:A$, we have that $a\in I$ is a closed proposition as it says that the $b_1,\cdots,b_l$ coordinates of $a$ are zero. 

Then $I$ is an ideal if and only if for any $a$ in the chosen basis of $A$ and any $i$ in the chosen basis of $I$ we have that $ai\in I$, which is a closed proposition.
\end{proof}

\begin{lemma}\label{severi-brauer-are-schemes}
For all $A:\AZ_n$ we define:
\[\RI(A) := \{I:\Gr_{n+1}(A)\ |\ I\ \mathrm{is\ a\ right\ ideal}\}\]
Then $\RI(A)$ is a scheme.
\end{lemma}

\begin{proof}
By \Cref{azumayas-are-finite-free} we have that $A$ is finite free as a module, so that by \Cref{grassmanians-are-schemes} we have that $\Gr_{n+1}(A)$ is a scheme, and then by \Cref{being-ideal-in-azumaya-closed} we have that $\RI(A)$ is closed in a scheme, so it is a scheme.
\end{proof}

\subsection{Quaternion algebras are Azumaya algebras}

In this section we assume $2\not=0$, and we take $T$ to be the étale sheafification.

\begin{definition}
Given $a,b:R^\times$, we define the quaternion algebra $Q(a,b)$ as the non-commutative algebra:
\[R[i,j]/(i^2=a,j^2=b,ij=-ji)\] 
\end{definition}

\begin{remark}
As a vector space, $Q(a,b)$ is of dimension $4$, generated by $1,i,j,ij$.
\end{remark}

\begin{remark}
By the change of variable $i\mapsto j$ and $j\mapsto i$ we get $Q(a,b) = Q(b,a)$.
\end{remark}

\begin{lemma}\label{quaternion-split}
For all $b:R^\times$, we have that $Q(1,b) = M_2(R)$.
\end{lemma}

\begin{proof}
We send $i$ to:
\[I = \begin{pmatrix}
1 & 0\\
0 & -1\\
\end{pmatrix}\]
and $j$ to:
\[J = \begin{pmatrix}
0 & b\\
1 & 0\\
\end{pmatrix}\]
Then $IJ$ is:
\[K = \begin{pmatrix}
0 & b\\
-1 & 0\\
\end{pmatrix}\]
It is easy to check this define an algebra morphism, and since $1,I,J,K$ form a basis of $M_2(R)$ the map is an isomorphism.
\end{proof}

\begin{lemma}\label{quaternion-change-variable}
For all $a,b,u,v:R^\times$, we have that $Q(a,b) = Q(u^2a,v^2b)$.
\end{lemma}

\begin{proof}
We use the variable change $i\mapsto ui$ and $j\mapsto vj$.
\end{proof}

\begin{lemma}
Given $a,b:R^\times$, we have that $Q(a,b)$ is an Azumaya algebra. 
\end{lemma}

\begin{proof}
We have that $Q(a,b)$ is finite free as a vector space so it is a sheaf. So $Q(a,b)$ being Azumaya is a sheaf and we can assume $\sqrt{a}$. Then by \Cref{quaternion-change-variable} we have $Q(1,b) = Q(\sqrt{a}^2,b) = Q(a,b)$ and we conclude by \Cref{quaternion-split}.
\end{proof}

\subsection{A remark on Azumaya algebras}

\begin{lemma}\label{MnR-endomorphism-multiplication}
For any $n:\N$, the map:
\[M_{n+1}(R)\otimes M_{n+1}(R)^{op}\to \mathrm{End}_R(M_{n+1}(R))\]
\[M\otimes N\mapsto (P\mapsto MPN)\]
is an equivalence.
\end{lemma}

\begin{proof}
Let us denote by $(E_{i,j})_{0\leq i,j\leq n}$ the canonical basis of $M_{n+1}(R)$. We consider the basis: 
\[(E_{i,j}\otimes E_{k,l})_{0\leq i,j,k,l\leq n}\] 
of $M_{n+1}(R)\otimes M_{n+1}(R)^{op}$, as well as the basis:
\[(C_{i,j,k,l})_{0\leq i,j,k,l\leq n}\] 
of $\mathrm{End}_R(M_{n+1}(R))$ where $C_{i,j,k,l}(E_{j,k}) = E_{i,l}$ and $C_{i,j,k,l}$ is null on other element of the basis. It is clear that the morphism sends one basis to the other, and that both algebras have the same multiplication table. 
\end{proof}

\begin{lemma}
Assume $A:\AZ_n$, then $A$ is finite free as a module and the map $A\otimes A^{op}\to \mathrm{End}_R(A)$ sending $a\otimes b$ to $c\mapsto acb$ is an equivalence.
\end{lemma}

\begin{proof}
The fact that $A$ is finite free is \Cref{azumayas-are-finite-free}. Then both $A\otimes A^{op}$ and $\mathrm{End}_R(A)$ are finite free modules and therefore are $T$-modal, so that the map being an equivalence is $T$-modal and when proving it we can assume $A=M_{n+1}(R)$. Then we conclude by \Cref{MnR-endomorphism-multiplication}.
\end{proof}


\subsection{The type $\SB_n$ of Severi-Brauer varieties}

\begin{definition}
A type $X$ is called a Severi-Brauer variety of dimension $n$ if $X$ is a sheaf and:
\[\propTrunc{X=\bP^n}_T\]
\end{definition}

We write $\SB_n$ the type of Severi-Brauer varieties of dimension $n$. We will see later that every Severi-Brauer variety is a scheme.

\begin{lemma}\label{right-ideal-of-matrices-are-projective}
Consider the map:
\[\delta:\bP^n \to \RI(M_{n+1}(R))\]
sending $(x_0:\cdots:x_n):\bP^n$ to:
\[\{M:M_n(R)\ |\ \forall i,j.\ x_i\cdot M_j = x_j\cdot M_i\}\]
where $M_i$ is the $i$-th line of $M$. Then $\delta$ is an equivalence.
\end{lemma}

\begin{proof}
Write $X=(x_0:\cdots:x_n)$. First we check $\delta$ is well defined. It is clear that for all $\lambda\not=0$ we have that:
\[\delta(\lambda X) = \delta(X)\]
and that $\delta(X)$ is a right ideal. To check the dimension assume $x_k\not=0$. Then $M\in\delta(X)$ if and only if for all $i$ we have that $M_i = \frac{x_i}{x_k} M_k$, which means giving $M\in\delta(X)$ is equivalent to giving $M_k$ in $R^{n+1}$, so $\delta(X)$ is free of dimension $n+1$.

Next we check injectivity. Assume given $(x_0:\cdots:x_n)$ and $(y_0:\cdots:y_n)$ in $\bP^n$ such that for all $M:M_n(R)$ we have:
\[(\forall i,j.\ x_i\cdot M_j = x_j\cdot M_i) \leftrightarrow (\forall i,j.\ y_i\cdot M_j = y_j\cdot M_i)\]
In particular considering the matrix $N$ such that $N_j = (y_j,\cdots,y_j)$ we get that:
\[\forall i,j.\ x_iy_j=x_jy_i\] 
so that:
\[(x_0:\cdots:x_n) = (y_0:\cdots:y_n)\]

Finally we check surjectivity. Assume $I:\RI(M_{n+1}(R))$, since $\propTrunc{I=R^{n+1}}$ we have $M\in I$ such that $M\not=0$, for example assume $M_{0,0}\not=0$. Then for all $k$ we have that:
\[ME_{0,k}\in I\]
meaning that we have $N^k\in I$ where:
\[N^k_{i,k} = M_{i,0}\]
and when when $j\not=k$
\[N^k_{i,j} = 0\]
Then since $M_{0,0}\not=0$ the matrices $N^k$ are linearly independent and since $I$ has dimension $n+1$, it is precisely the ideal spanned by the $N^k$. But this ideal is $\delta(M_{0,0}:\cdots:M_{n,0})$.
\end{proof}

\begin{lemma}
If $A$ is an Azumaya algebra, then $\RI(A)$ is a Severi-Brauer variety.
\end{lemma}

\begin{proof}
By \Cref{severi-brauer-are-schemes} and the assumption that schemes are sheaves, we have that $\RI(A)$ is a sheaf. Then to prove:
\[\propTrunc{A=M_{n+1}(R)}_T \to \propTrunc{\RI(A)=\bP^n}_T\]
it is enough to prove:
\[\RI(M_{n+1}(R)) = \bP^n\]
which is \Cref{right-ideal-of-matrices-are-projective}.
\end{proof}

\subsection{Conics are Severi-Brauer varieties}

In this section we assume $2\not=0$, and we take $T$ to be the étale sheafification.

\begin{definition}
Given $a,b:R^\times$, we define the conic $C(a,b)$ as the set of $[x:y:z]:\bP^2$ such that $x^2=ay^2+bz^2$.
\end{definition}

\begin{lemma}\label{pointed-conics-projective}
Assume $a,b:R^\times$ such that $\propTrunc{C(a,b)}$, then $\propTrunc{C(a,b)=\bP^1}$.
\end{lemma}

\begin{proof}
Let us assume $x_0,y_0,z_0$ such that $x_0^2 = ay_0^2+bz_0^2$. We can assume $x_0\not=0$ by possibly considering $C(a,b) = C(\frac{1}{a},-\frac{b}{a})$. Then we can clearly assume $x_0=1$ without loss of generality, so that $ay_0 + bz_0 = 1$.

 Let us consider the map:
\[\psi:\bP^1\to \bP^2\]
\[[u:v] \mapsto [au^2+bv^2: y_0(au^2-bv^2) + 2buvz_0 : z_0(au^2-bv^2) - 2auvy_0]\]

We want to define $\phi$ inverse to $\phi$. Assume $[x:y:z]:\bP^2$ such that $x^2=ay^2+bz^2$. 

Let us proof that either $x+ay_0y+bz_0z$ or $x-ay_0y-bz_0z$ is invertible, to do this it is enough to prove that either $x$ or $ay_0y+bz_0z$ is invertible. Assume $x=0$ and $ay_0y+bz_0z=0$, then we have a contradiction. Indeed $y$ or $z$ is invertible and $ay^2+bz^2=0$, so that $y$ and $z$ are invertible and $b = -a\frac{y^2}{z^2}$ and $y_0z=yz_0$. Moreover $y_0$ or $z_0$ is invertible, so that both $y_0$ and $z_0$ are invertible and $\frac{y}{z} = \frac{y_0}{z_0}$ which means $ay^2+bz^2=0$ implies $ay_0^2+bz_0^2=0$, a contradiction. 

If $x + ay_0y + bz_0z$ is invertible we define $\phi([x,y,z]) = [1:\frac{a(z_0y-y_0z)}{x + ay_0y + bz_0z}]$.

If $x - ay_0y - bz_0z$ is invertible we define $\phi([x,y,z]) = [\frac{b(z_0y-y_0z)}{x - ay_0y - bz_0z}:1]$.

This is well defined as if both are invertible then:
\[\frac{b(z_0y-y_0z)}{x - ay_0y - bz_0z}\times\frac{a(z_0y-y_0z)}{x + ay_0y + bz_0z} = \frac{ab(z_0y-y_0z)^2}{x^2 - (ay_0y + bz_0z)^2} = 1\]
because:
\[x^2 - (ay_0y + bz_0z)^2 = (ay^2+bz^2)(ay_0^2+bz_0^2) - (ay_0y + bz_0z)^2 = ab(z_0y-y_0z)^2\]
We omit the verification that this $\phi$ is indeed an inverse to $\psi$.
\end{proof}

\begin{lemma}\label{conic-one-split}
Assume given $b:R^\times$, then $\propTrunc{C(1,b) = \bP^1}$.
\end{lemma}

\begin{proof}
We just apply \Cref{pointed-conics-projective} to the point $[1,1,0]:C(1,b)$.
\end{proof}

\begin{lemma}\label{conic-change-variable}
Given $a,b,u,v:R^\times$ we have that $C(a,b) = C(u^2a,v^2b)$.
\end{lemma}

\begin{proof}
Consider the change of variable $y\mapsto uy$ and $z\mapsto vz$.
\end{proof}

\begin{lemma}
Assume given $a,b:R^\times$, then $C(a,b)$ is a Severi-Brauer variety.
\end{lemma}

\begin{proof}
Since a sheaf being a Severi-Brauer variety is an \'etale sheaf, we can assume $\sqrt{a}$ . Then by \Cref{conic-change-variable} we have that $C(1,b) = C(\sqrt{a}^2,b)=C(a,b)$ and we conclude by \Cref{conic-one-split}.
\end{proof}

\begin{remark}
We will see in \Cref{chatelet-theorem} that any merely inhabited Severi-Brauer variety is a projective space. 
\end{remark}

\section{The equivalence $\AZ_n\simeq \SB_n$ and Ch\^atelet's Theorem}
\subsection{Generalities on delooping in $T$-sheaves}

\begin{definition}
A type $A$ is $T$-connected if:
\[\forall(x,y:A).\ \propTrunc{x=y}_T\]
\end{definition}

The key intuition for the next lemma is that both $A$ and $B$ are deloopings of the same group in the topos of sheaves.

\begin{lemma}\label{deloopings-equivalence}
Assume $A,B$ pointed $T$-connected sheaves. Let $f:A\to B$ be a pointed map inducing an equivalence:
\[\Omega f : \Omega A \simeq \Omega B\]
Then $f$ is an equivalence.
\end{lemma}

\begin{proof}
First we prove that $f$ is an embedding. We have to prove that for all $x,y:A$ the map:
\[\mathrm{ap}_f : x=y \to f(x)=f(y)\]
is an equivalence. Since $A$ and $B$ are sheaves so are their identity types, so $\mathrm{ap}_f$ being an equivalence is a sheaf, so by $T$-connectedness of $A$ we can assume $x$ and $y$ are the basepoint, in which case it is part of the hypothesis.

Now we prove it is surjective, indeed for any $x:B$ we have that $\fib_f(x)$ is a sheaf and a proposition so when proving it is inhabited we can assume $x$ is the basepoint of $B$ and give the basepoint of $A$ as antecedent.
\end{proof}

\subsection{Both $\Aut(M_{n+1}(R))$ and $\Aut(\bP^n)$ are $\PGL_{n+1}(R)$}

We need an well-known algebra result.

\begin{lemma}\label{finite-projective-free}
Let $M$ be a finite projective module, then $M$ is finite free.
\end{lemma}

\begin{proof}
This is IX.2.2 in \cite{lombardi-quitte}, it crucially relies on $R$ being local.
\end{proof}

\begin{lemma}\label{fundamental-system-matrices}
Assume $e_{i,j}:M_{n+1}(R)$ for $0\leq i,j\leq n$ such that:
\[e_{i,j}e_{k,l} = \delta_{j,k}e_{i,l}\]
where $\delta_{j,k} = 1$ if $j=k$ and $0$ otherwise. Moreover assume:
\[e_{0,0}+\cdots+e_{n,n}=1\]
Then there exists $P:GL_{n+1}(R)$ such that:
\[e_{i,j} = PE_{i,j}P^{-1}\]
\end{lemma}

\begin{proof}
We define $e_i = e_{i,i}$, then $e_0+\cdots+e_n = 1$, for all $i$ we have $e_i^2=e_i$ and for all $i\not=j$ we have that $e_ie_j=0$. From this we get:
\[R^{n+1} = V_0\oplus\cdots\oplus V_n\]
where:
\[V_i = \{x\ |\ e_i(x)=x\}\]
and:
\[e_{i,j}:V_j\simeq V_i\]

As a direct summand of a free module we have that $V_0$ is projective, and since $V_0 = e_{0}(R^{n+1})$ we have that $V_0$ is finitely generated, so by \Cref{finite-projective-free} it is finite free. From $V_0^{n+1}=R^{n+1}$ we get that $\propTrunc{V_0=R}$, and therefore that $\propTrunc{V_i=R}$ for all $i$.

Then we choose $v_0$ generating $V_0$ and define $v_i = e_{i,0}(v_0)$ so that $v_i$ generates $V_i$ because $e_{i,0}:V_0\simeq V_i$. We get a basis $v_0,\cdots,v_n$ of $R^{n+1}$.

Let $u_0,\cdots,u_n$ be the canonical basis of $R^{n+1}$ and define $P:GL_{n+1}(R)$ by sending $u_i$ to $v_i$. Then for all $v_k$ we have that:
\[e_{i,j}v_k = PE_{i,j}P^{-1}v_k\]
so we can conclude.
\end{proof}

\begin{proposition}\label{Aut-MnR-PGL}
The map:
\[\alpha:\PGL_{n+1}(R)\to\Aut(M_{n+1}(R))\]
\[P\mapsto (M\mapsto PMP^{-1})\]
is an equivalence.
\end{proposition}

\begin{proof}
It is clearly a group morphism. 

For injectivity we just need to check that if for all $M:M_{n+1}(R)$ we have $PMP^{-1}=M$ then there exists $\lambda\not=0$ such that $P=\lambda I_{n+1}$. We deduce this from $Pe_{i,j} = e_{i,j}P$ and $P$ invertible.

For surjectivity consider $e_{i,j}=\sigma(E_{i,j})$, we can apply \Cref{fundamental-system-matrices} to get $P:GL_{n+1}(R)$ such that:
\[\sigma(E_{i,j}) = PE_{i,j}P^{-1}\]
from which we conclude that for all $M:M_{n+1}(R)$ we have that:
\[\sigma(M) = PMP^{-1}\]
as desired.
\end{proof}

\begin{proposition}\label{Aut-Pn-PGL}
The map:
\[\beta:\PGL_{n+1}(R)\to\Aut(\bP^n)\]
\[X\mapsto PX\]
is an equivalence.
\end{proposition}

\begin{proof}
This is the main result from \cite{sag-projective}.
\end{proof}

\subsection{The Severi-Brauer construction is an equivalence}

\begin{proposition}\label{right-ideal-is-equivalence}
The map:
\[\RI:\AZ_n\to\SB_n\]
is an equivalence.
\end{proposition}

\begin{proof}
By \Cref{deloopings-equivalence} it is enough to prove that the top map in the triangle:
\begin{center}
\begin{tikzcd}
\Aut(M_n(R))\ar[rr,"\Omega\RI"] & & \Aut(\bP^n) \\
& \PGL_{n+1}(R)\ar[ru,swap,"\beta"]\ar[lu,"\alpha"] & \\
\end{tikzcd}
\end{center}
is an equivalence. But since the two other maps in the triangle are equivalences by \Cref{Aut-MnR-PGL} and \Cref{Aut-Pn-PGL}, it is enough to prove that the triangle commutes. To do this we need to check that for all $P:\PGL_{n+1}(R)$ we have that:
\[\delta^{-1}\circ \mathrm{ap}_\RI(\alpha(P))\circ\delta = \beta(P)\]
in $\Aut(\bP^n)$, with $\delta$ defined in \Cref{right-ideal-of-matrices-are-projective}. So we need to prove the following square commutes:
\begin{center}
\begin{tikzcd}
\RI(M_{n+1}(R))\ar[rr,"I\mapsto PIP^{-1}"]&& \RI(M_{n+1}(R)) \\
\bP^n\ar[u,"\delta"]\ar[rr,swap,"X\mapsto PX"]&& \bP^n\ar[u,swap,"\delta"] 
\end{tikzcd}
\end{center}
where $\mathrm{ap}_\RI$ was computed using path induction.

We see that:
\[\delta(Y) = \{M:M_n(R)\ |\ \forall A,B:R^{n+1}.\ A^tX\cdot B^tM = B^tX\cdot A^tM\}\]
To check that:
\[P\delta(X)P^{-1} = \delta(PX)\]
we just need to check an inclusion as both are finite free modules of the same dimension. Assume $M\in\delta(X)$, to check that $PMP^{-1}\in\delta(PX)$ we need to check that for all $A,B:R^{n+1}$ we have that:
\[A^tPX\cdot B^tPMP^{-1} = B^tX\cdot A^tPMP^{-1}\]
but since $M\in\delta(X)$ we have that:
\[(P^tA)^tX\cdot (P^tB)^tM = (P^tB)^tX\cdot (P^tA)^tM\]
which gives us what we want.
\end{proof}

\begin{remark}
By \Cref{severi-brauer-are-schemes} and \Cref{right-ideal-is-equivalence} we can conclude than any Severi-Brauer variety is a scheme. This was not clear a priori because being a scheme is not modal.
\end{remark}

\begin{remark}\label{severi-brauer-independent-modality}
By \Cref{right-ideal-is-equivalence} and \Cref{azumaya-independent-modality} and we can conclude that a type $X$ being a Severi-Brauer variety for any modality $T$ such that:
\begin{itemize}
\item Schemes are $T$-modal.
\item The type of finite free modules is $T$-modal.
\item $T$-modal types are étale sheaves.
\end{itemize}
is equivalent to $X$ being a Severi-Brauer variety for the étale topology. In particular Severi-Brauer varieties do not depend on the choice of such a $T$.
\end{remark}

\subsection{Ch\^atelet's Theorem}

\begin{lemma}\label{azumaya-with-right-ideal}
Assume $A:\AZ_n$ with $I:\RI(A)$, then:
\[A = \mathrm{End}_R(I)^{op}\]
\end{lemma}

\begin{proof}
Since $I$ is an ideal, there is a canonical map of algebra:
\[A \to\mathrm{End}_R(I)^{op}\]
Since both algebras are sheaves (indeed $\propTrunc{I=R^{n+1}}$ implies $I$ is a sheaf), this map being an equivalence is a sheaf so we can assume $A=M_{n+1}(R)$.

By \Cref{right-ideal-of-matrices-are-projective} we can assume $X=(x_0:\cdots:x_n):\bP^n$ such that $I=\delta(X)$. There is a $k$ such that $x_k\not=0$, so we can assume $x_k=1$, then we have an isomorphism:
\[\theta:R^{n+1}\to I\]
sending $Y:R^{n+1}$ to the matrix $M$ with its $i$-th line $M_i=x_iY$. Then for all $M:M_n(R)$ we have a commutative square:
\begin{center}
\begin{tikzcd}
I\ar[rr,"N\mapsto NM"] && I \\
R^{n+1}\ar[u,"\theta"]\ar[rr,swap,"X\mapsto M^tX"] && R^{n+1}\ar[u,swap,"\theta"]\\
\end{tikzcd}
\end{center}
meaning the natural map:
\[ M_{n+1}(R)\to \mathrm{End}_R(I)^{op}\]
sends $M$ to $\delta^{-1}\circ M\circ\delta$, so it is an equivalence.
\end{proof}

\begin{theorem}[Ch\^atelet's Theorem]\label{chatelet-theorem}
Assume $X:\SB_n$, then:
\[\propTrunc{X}\to\propTrunc{X=\bP^n}\]
\end{theorem}

\begin{proof}
By \Cref{right-ideal-is-equivalence} we can assume $X=\RI(A)$ for some $A:\AZ_n$. Then we can assume $I:\RI(A)$, so that by \Cref{azumaya-with-right-ideal} we have that:
\[A=\mathrm{End}_R(I)^{op}\]
Since we merely have that $I=R^{n+1}$, we merely have:
\[A = M_{n+1}(R)^{op} = M_{n+1}(R)\]
Applying \Cref{right-ideal-of-matrices-are-projective} we merely conclude:
\[X=\RI(A)=\RI(M_{n+1}(R)) = \bP^n\]
\end{proof}


\begin{thebibliography}{10}

\bibitem{azumaya51}
Gor\^o Azumaya.
\newblock On maximally central algebras.
\newblock {\em Nagoya Math. J.}, 2:119--150, 1951.

\bibitem{chatelet44}
Fran\c~cois Ch\^atelet.
\newblock {\em Variations sur un th\`eme de {H}. {P}oincar\'e}.
\newblock Annales scientifiques de l'\'Ecole Normale Sup\'erieure, S\'erie 3,
  Volume 61, 1944.

\bibitem{draft}
Felix {Cherubini}, Thierry {Coquand}, and Matthias {Hutzler}.
\newblock A foundation for synthetic algebraic geometry.
\newblock {\em Mathematical Structures in Computer Science}, 34(9):1008--1053,
  2024.

\bibitem{sag-projective}
Felix Cherubini, Thierry Coquand, Matthias Hutzler, and David Wärn.
\newblock Projective space in synthetic algebraic geometry.
\newblock 2024.

\bibitem{colliot88}
Jean-Louis Colliot-Th\'el\`ene.
\newblock Les grands th\`emes de {F}ran\c cois {C}h\^atelet.
\newblock {\em Enseign. Math. (2)}, 34(3-4):387--405, 1988.

\bibitem{coqazumaya}
Thierry Coquand, Henri Lombardi, and Stefan Neuwirth.
\newblock Constructive remarks on azumaya algebra, 2023.

\bibitem{gille2017central}
Philippe Gille and Tam{\'a}s Szamuely.
\newblock {\em Central simple algebras and Galois cohomology}, volume 165.
\newblock Cambridge University Press, 2017.

\bibitem{grothendieck68}
Alexander Grothendieck.
\newblock Le groupe de {B}rauer. {I}. {A}lg\`ebres d'{A}zumaya et
  interpr\'etations diverses.
\newblock In {\em Dix expos\'es sur la cohomologie des sch\'emas}, volume~3 of
  {\em Adv. Stud. Pure Math.}, pages 46--66. North-Holland, Amsterdam, 1968.

\bibitem{lombardi-quitte}
Henri Lombardi and Claude Quitt{\'{e}}.
\newblock {\em Commutative Algebra: Constructive Methods}.
\newblock Springer Netherlands, 2015.

\bibitem{hott}
The Univalent~Foundations Program.
\newblock {\em Homotopy type theory: Univalent foundations of mathematics}.
\newblock \url{https://homotopytypetheory.org/book}, 2013.

\bibitem{Quirin16}
Kevin Quirin.
\newblock {\em Lawvere-Tierney sheafification in Homotopy Type Theory.
  (Faisceautisation de Lawvere-Tierney en th{\'{e}}orie des types
  homotopiques)}.
\newblock PhD thesis, {\'{E}}cole des mines de Nantes, France, 2016.

\bibitem{modalities}
Egbert Rijke, Michael Shulman, and Bas Spitters.
\newblock {Modalities in homotopy type theory}.
\newblock {\em {Logical Methods in Computer Science}}, {Volume 16, Issue 1},
  January 2020.

\bibitem{serre62}
Jean-Pierre Serre.
\newblock Corps locaux.
\newblock Publications de l'{Institut} de {Math{\'e}matique} de
  l'{Universit{\'e}} de {Nancago}. 8; {Actualit{\'e}s} {Scientifiques} et
  {Industrielles}. 1296. {Paris}: {Hermann} \& {Cie}. 243 p., 1962.

\bibitem{wraith79}
G.~C. Wraith.
\newblock Generic {Galois} theory of local rings.
\newblock Applications of sheaves, {Proc}. {Res}. {Symp}., {Durham} 1977,
  {Lect}. {Notes} {Math}. 753, 739-767, 1979.

\end{thebibliography}
\end{document}